\documentclass[letterpaper, 10 pt, conference]{ieeeconf}
\usepackage{ifpdf, flushend}
\usepackage{graphicx}
\usepackage{epstopdf}
\graphicspath{{../}}
\usepackage{multirow}
\usepackage{float}
\usepackage[cmex10]{amsmath}
\usepackage{amssymb}
\usepackage{array}
\usepackage{mdwmath}
\usepackage{mdwtab}
\usepackage{eqparbox}
\usepackage{url}
\usepackage{nomencl}
\usepackage{cite}
\usepackage[ruled,vlined]{algorithm2e}
\usepackage{makeidx}
\usepackage{ifthen}
\usepackage{subcaption}
\usepackage{pdflscape}
\usepackage[colorlinks=true, bookmarks=true]{hyperref}
\AtBeginDocument{%
 \hypersetup{
   citecolor=red,
   linkcolor=blue,  
   urlcolor=blue}}

\makeatletter
\def\munderbar#1{\underline{\sbox\tw@{$#1$}\dp\tw@\z@\box\tw@}}
\makeatother
\def\sw{{\text{sw}}}
\def\ex{{\text{ex}}}
\def\bl{{\text{bl}}}

\newcommand{\mc}[1]{\mathcal{#1}}

\newtheorem{lemma}{Lemma}
\newtheorem{theorem}{Theorem}

\IEEEoverridecommandlockouts    
\title{\LARGE \bf
Efficient Graph Partitioning under Resource Constraints: \\ A Cutting-Plane Framework for Distribution Grids
}

\author{Duong Thuy Anh Nguyen, Harsha Nagarajan, Robert Ferrando, Russell Bent, David Fobes
\thanks{D. Nguyen (\href{mailto:dtnguy52@asu.edu}{dtnguy52@asu.edu}) is with the School of Electrical, Computer and Energy
Engineering, Arizona State University, Tempe, AZ, USA. H. Nagarajan (\href{mailto:harsha@lanl.gov}{harsha@lanl.gov}) and R. Bent (\href{mailto:rbent@lanl.gov}{rbent@lanl.gov}) are with Applied Mathematics and Plasma Physics (T-5), Los Alamos National Laboratory (LANL), Los Alamos, NM, USA.
R. Ferrando (\href{mailto:rferrando@lanl.gov}{rferrando@lanl.gov}) and D. Fobes (\href{mailto:dfobes@lanl.gov}{dfobes@lanl.gov}) are with Computational Intelligence and Modeling (A-1), LANL, Los Alamos, NM, USA. 
This work was supported by the U.S. Department of Energy, Office of Electricity's Microgrid Research and Development (R\&D) Program under the ``Dynamic Microgrids for Large-Scale DER Integration and Electrification (DynaGrid)'' project, and LANL's Lab Directed R\&D program under the project ``20230091ER: Learning to Accelerate Global Solutions for Non-convex Optimization''.
}
}
\begin{document}
\maketitle

\begin{abstract}
This paper presents an optimal network topology control framework using cutting-plane methods for efficient network partitioning with controllable edges. The objective is to enable real-time reconfiguration of interconnected subnetworks while ensuring radial connectivity, resource feasibility, and structured leader allocation, which provide structural foundations for distributed controllability, stability, and coordination. The problem is formulated as a mixed-integer program that integrates graph-theoretic constraints, resource flow, and network structural properties to enforce an operational hierarchy. To address the combinatorial complexity of cycle elimination and leader assignment, we propose an iterative cutting-plane framework that ensures convergence to an optimal and feasible network topology. Theoretical guarantees on optimality preservation, feasibility, and convergence are established, ensuring systematic elimination of infeasible configurations while preserving the required coordination structure. Simulations on a modified Iowa 240-bus power distribution grid demonstrate the framework's effectiveness in network reconfiguration under resource constraints. The approach achieves a median speedup of 65.4$\times$ and a best-case speedup of over 78$\times$ in a 46-switch configuration.
\end{abstract}

\printnomenclature
\allowdisplaybreaks

\section{Introduction} \label{intro}
Network partitioning has long been recognized as a fundamental strategy for managing complexity in large-scale, interconnected systems such as power grids \cite{chen2022distribution}, multi-agent systems \cite{nguyen2023acceleratedabpushpull}, and sensor networks \cite{akyildiz2002wireless}. Partitioning a global network into smaller, self-contained subnetworks localizes control, optimizes resource distribution, and isolates critical infrastructure during emergencies or overloads \cite{baran1989optim, kazm2026partitioning, bertsekas1992data}. In distribution grids, this topology and resource allocation underpin stable islanded operation. Motivated by the broad applicability of these strategies, this paper focuses on the optimal network partitioning problem.

Despite the wide-ranging applications, achieving both optimality and scalability remains a fundamental challenge due to the \textit{combinatorial explosion} from interdependent \textit{topological} and \textit{operational} decisions. In power distribution grids, ensuring each microgrid (a self-contained small-scale grid) remains acyclic (radial) leads to an exponential number of constraints even for moderate-sized systems \cite{ImportanceOfRadiality,lei2020radiality,ahmed2008cycle,yang2019towards}. Similarly, in multi-agent networks, forming stable coalitions that satisfy task requirements and communication constraints is known to be NP-hard \cite{accikmecse2012observers}. While heuristic clustering protocols \cite{akyildiz2002wireless, Younis2004HEED} reduce computational overhead, they often lack feasibility and performance guarantees.

A common approach is to formulate a monolithic mixed-integer program (MIP) that explicitly encodes all constraints. While such monolithic formulations guarantee optimality in theory, they become computationally intractable as network size increases \cite{chen2020strong}. Enforcing topology properties like radiality can require an exponential number of cycle-prevention constraints \cite{padberg1991branch}, and integrating resource allocation constraints further increases complexity \cite{chen2020strong, MergingMicrogrids,moring2024robust}. Early reconfiguration methods \cite{baran1989optim, merlin1975search} and subsequent advances in multi-commodity flow \cite{lei2020radiality} or sub-tour elimination \cite{ahmed2008cycle} remain limited by exponential growth in binary variables \cite{padberg1991branch}.

Given this intractability, metaheuristic local search algorithms (e.g., genetic algorithms \cite{coello2007evolutionary}, particle swarm or local search) are often employed. However, these methods often rely on problem-specific parameter tuning and heuristics without rigorous optimality or feasibility guarantees for complex constraints. Consequently, there remains a critical need for \textit{scalable} methods that are either \textit{exact} or \textit{provably near-optimal}, and can systematically integrate the full spectrum of constraints in modern network problems.

In recent years, \textit{decomposition-based} methods, particularly cutting-plane algorithms, have garnered significant interest for high-dimensional MIPs \cite{magnanti1984networks, padberg1991branch, chen2020strong}. By generating constraints only to exclude detected infeasibilities, such as cycles or subnetworks with insufficient resources, iterative frameworks often achieve substantial computational speedups compared to monolithic formulations \cite{benders1962partitioning, magnanti1984networks, contreras2011benders}. However, their application to \textit{integrated network partitioning} with complex interdependencies between topological and operational constraints remains relatively underexplored.

Motivated by these limitations, this paper introduces a \textit{novel cutting-plane framework for fast, optimal network partitioning} that integrates topological, resource, and hierarchical constraints. To address the combinatorial complexity of these requirements, we present a comprehensive set of mixed-integer linear constraints, leveraging a directed multi-commodity flow-based approach \cite{lei2020radiality} to enforce radiality and a coloring-based method \cite{moring2024robust} to ensure the correct assignment of hierarchy within subnetworks. Although these formulations provide a rigorous mathematical foundation, their direct implementation results in computationally expensive optimization, especially as network size increases.
Our key contribution is an iterative cutting-plane algorithm that enforces these constraints dynamically. We develop \textit{novel cutting-plane constraints (or cuts)} that systematically and efficiently refine the solution space. The algorithm begins by solving a relaxed partitioning model. If infeasibilities are detected, appropriate cutting-plane constraints are dynamically introduced in subsequent iterations. Notably, we formally establish theoretical guarantees that the proposed algorithm converges to the same optimal solution as the full monolithic MIP but with significantly reduced computational effort.

\textit{Contributions} are summarized as follows. \textit{(i)} We introduce a cutting-plane framework for optimal network partitioning that integrates graph-theoretic constraints, resource flow, and structured leader allocation while ensuring topological and operational feasibility. Unlike \cite{moring2024robust}, which retains coloring constraints in the master problem, our framework accelerates the graph-partitioning problem itself. \textit{(ii)} We develop iterative cutting-plane constraints that jointly enforce radiality and leader-coloring requirements, adapting the coloring formulation from prior work while introducing a new cut-based treatment that significantly reduces computational complexity compared with monolithic MIPs. \textit{(iii)} We theoretically guarantee optimality, feasibility, and convergence, and validate the framework on an Iowa 240-bus power system.

\section{Network Model} \label{sec:Network}
Consider a network represented as a directed graph $\mathbb{G} = (\mathcal{N}, \mathcal{E})$, where $\mathcal{N} = \{1, 2, \dots, N\}$ is the set of nodes, representing individual components such as devices, agents, subsystems, or entities. The set of edges $\mathcal{E} \subseteq \mathcal{N} \times \mathcal{N}$ represents the connections between nodes, which may include communication links, physical connections, power lines, or transportation pathways. Each edge $(i, j) \in \mathcal{E}$ is oriented from node $i$ to node $j$, indicating the permissible direction of flow, whether it be data, power, or other resources. A subset of $\mathcal{E}$, denoted by $\mathcal{E}^{\mathrm{sw}}\subseteq \mathcal{E}$, comprises the \textbf{controllable edges} whose open/closed state can be selected to partition the network into subnetworks. An undirected network is a special case of this formulation, where each connection is bidirectional. The network structure is illustrated in Fig.~\ref{fig:network}.

\begin{figure}[t!]
    \vspace{.2cm}
   \centering
   \includegraphics[width=0.49\textwidth]{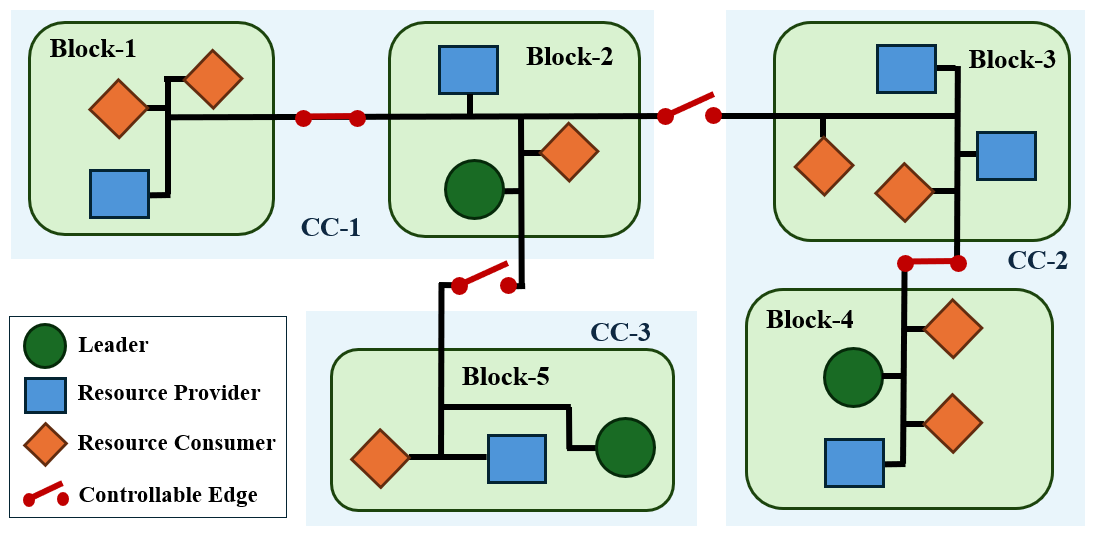}
   \caption{Prototype network.}
   \label{fig:network}
\end{figure}

This paper addresses optimal network reconfiguration by controlling $\mathcal{E}^{\mathrm{sw}}$, enabling transitions between topological states in power grids, multi-agent systems, and distributed computing. For a fixed switch assignment $\boldsymbol{z}^{\mathrm{sw}}$, $\mathcal{C}(\boldsymbol{z}^{\mathrm{sw}})$ denotes the set of connected components induced by the closed controllable edges. The formulation is static and flow-based: no system dynamics are modeled, and all operational constraints are steady-state.

\subsection{Base Blocks}
Define $\mathcal{B}$ as the set of disjoint \textit{base blocks} $\ell \in \mathcal{B}$. Each block represents the smallest predefined subnetwork, a subset of nodes that remains connected when all controllable edges are open.
To model the activation state of each base block $\ell \in \mathcal{B}$, we introduce the binary decision variable:
\begin{align}
   z^{\mathrm{bl}}_{\ell} =
   \begin{cases}
   1, & \text{if block } \ell \in \mathcal{B} \text{ is active}, \\
   0, & \text{otherwise}.
   \end{cases}
\end{align}
Here, $z^{\mathrm{bl}}_{\ell} = 1$ indicates that the constituent nodes of block $\ell$ are operational and functional within the network.

\subsection{Controllable Edges}
The subset $\mathcal{E}^{\sw} \subseteq \mathcal{E}$ denotes \textit{controllable edges} (or ``switches") that enable network reconfiguration. For each controllable edge $(i, j) \in \mathcal{E}^{\mathrm{sw}}$, we define a decision variable:
\begin{align}
    z_{ij}^{\mathrm{sw}}
    =
    \begin{cases}
      1, & \text{if the edge is closed (active)}, \\
      0, & \text{if the edge is open}.
    \end{cases}
\end{align}

Closing a switch ($z_{ij}^{\mathrm{sw}} = 1$) merges base blocks into larger \textit{connected components}, defined as maximal subsets of nodes where every pair is linked by a sequence of active edges. Conversely, an open switch ($z_{ij}^{\mathrm{sw}} = 0$) prevents interaction between blocks, preserving their independent operation. This control over $\mathcal{E}^{\mathrm{sw}}$ enables \textit{adaptive network partitioning} for objectives such as fault isolation and load balancing in power grids, or dynamic coalition formation in multi-agent systems.

\subsection{Node Roles and Leader Nodes}
We partition the node set $\mathcal{N}$ into three disjoint subsets: resource providers $\mathcal{P}$, resource consumers $\mathcal{M}$, and intermediary nodes $\mathcal{I}$. Providers supply essential resources (e.g., power, data), consumers utilize them, and intermediaries offer auxiliary routing or communication functionality. By construction, these sets satisfy
\begin{align}
   \mathcal{N} = \mathcal{P} \cup \mathcal{M} \cup \mathcal{I}, \quad \mathcal{P} \cap \mathcal{M} = \mathcal{P} \cap \mathcal{I} = \mathcal{M} \cap \mathcal{I} = \varnothing.
\end{align}
Similarly, for each base block $\ell \in \mathcal{B}$, the local node set $\mathcal{N}_{\ell}$ is partitioned into $\mathcal{P}_{\ell}$, $\mathcal{M}_{\ell}$, and $\mathcal{I}_{\ell}$.

Among the resource providers $\mathcal{P}$, certain nodes are designated as \textit{leaders} to manage essential control and coordination functions. For example, a leader may represent a grid-forming inverter in distribution grids, a primary router in communication systems, or a coordination hub in multi-agent networks. We introduce a binary decision variable $z_{g}^{\mathrm{ldr}}$ for each $g \in \mathcal{P}$, where $z_{g}^{\mathrm{ldr}} = 1$ if node $g$ is designated as a leader. 

\section{Optimal Network Partitioning Problem} \label{sec:prob}

We present a high-level formulation of the \textit{Optimal Network Partitioning Problem}, which governs the activation of controllable edges and blocks. Given the network of Section~\ref{sec:Network}, the goal is to partition it into subnetworks that optimize the system objectives subject to structural and operational constraints.

Let $\boldsymbol{x} = (\boldsymbol{z}^{\mathrm{sw}}, \boldsymbol{z}^{\mathrm{bl}}, \boldsymbol{z}^{\mathrm{ldr}}, \boldsymbol{r}^{\mathrm{gen}}, \boldsymbol{r}^{\mathrm{flow}}, \dots)$ denote the decision vector. This includes binary variables for topology and leadership status, as well as continuous variables for resource generation, distribution, and other domain-specific constraints. The objective function, $\mathcal{O}(\boldsymbol{x})$, captures domain-specific goals such as minimizing operational costs, balancing resource loads, or maximizing service coverage. The general optimization problem, $\Omega$, is formulated as:
\begin{align} \label{eq-OptimalNetworkPartitioningProblem}
   \min_{\boldsymbol{x}} \quad & \mathcal{O}(\boldsymbol{x}) \\
   \text{s.t.} \quad & \mathcal{C}^1(\boldsymbol{x}) ~ \text{(network configuration constraints)},  \nonumber\\
   & \mathcal{C}^2(\boldsymbol{x}) ~ \text{(resource and capacity constraints)}, \nonumber\\
   & \mathcal{C}^3(\boldsymbol{x}) ~ \text{(operational and functional role constraints)}, \nonumber\\
   & \mathcal{C}^4(\boldsymbol{x}) ~ \text{(domain-specific constraints)}, \nonumber\\
   & \mathcal{C}^5(\boldsymbol{x}) ~ \text{(variable domains constraints)}. \nonumber
\end{align}
The feasible region is $\Xi = \{\boldsymbol{x} : \mathcal{C}^1(\boldsymbol{x}), \dots, \mathcal{C}^5(\boldsymbol{x}) \text{ hold}\}$, which ensures a valid and functional network partition.

\section{Network Partitioning Constraints}
\label{sec:constraints}

This section defines the feasibility constraints for valid network partitioning, including connectivity, resource allocation, and hierarchical coordination. Because these requirements introduce significant combinatorial complexity in large-scale systems, enforcing them directly is often impractical. We therefore propose an iterative cutting-plane framework that enforces them selectively, ensuring computational efficiency and convergence to an optimal solution.

\subsection{Network Configuration Constraints}
To ensure network efficiency and stability, each CC must satisfy structural requirements, most notably \textit{radiality}. Enforcing a tree-like topology is essential in power distribution systems to support standard protection-coordination practices and avoid redundant loops.

For a given configuration $\boldsymbol{z}^{\mathrm{sw}}$, each CC $\ell \in \mathcal{C}(\boldsymbol{z}^{\mathrm{sw}})$ must form a spanning tree. Define $\mathcal{N}_{\ell}$ as the set of all nodes within the CC $\ell$, and let $\mathcal{E}_{\ell}$ denote the set of edges that belong to CC $\ell$. Following \cite{lei2020radiality}, we adopt a directed multi-commodity flow-based formulation. We designate a reference node $i_r \in \mathcal{N}_{\ell}$ as the root. For every other node $k \in \mathcal{N}_{\ell} \setminus \{i_r\}$, we introduce a fictitious commodity $k$ and flow variables $f^k_{ij}$. The binary variable $\lambda_{ij}$ indicates if arc $(i, j)$ is part of the directed spanning tree, while $\beta_{ij}$ (corresponding to the switches $z_{ij}^{\mathrm{sw}}$) denotes the undirected edge status ($\beta_{ij} = 1$ if active). The following constraints enforce radiality for each $\ell \in \mathcal{C}(\boldsymbol{z}^{\mathrm{sw}})$.

\subsubsection{Flow Conservation}
Each commodity $k$ originates at the root $i_r$ and is consumed at node $k$:
\begin{align}
   \sum_{j:(j, i) \in \mathcal{E}_{\ell}} f^{k}_{ji} - \sum_{j:(i, j) \in \mathcal{E}_{\ell}} f^{k}_{ij} = 
   \begin{cases} 
   -1, & \text{if } i = i_r \\
   1, & \text{if } i = k \\
   0, & \text{otherwise}
   \end{cases} \label{eq:radial_flow_conservation}
\end{align}

\subsubsection{Topological Consistency}
Flow can only exist on active directed arcs, and the number of active edges must satisfy the tree property. Given $\lambda_{ij}, \lambda_{ji}, \beta_{ij} \in \{0, 1\}, \forall (i, j) \in \mathcal{E}_{\ell}$,
\begin{align}
   &0 \leqslant f^{k}_{ij} \leqslant \lambda_{ij}, \quad \forall (i,j) \in \mathcal{E}_{\ell}, k \in \mathcal{N}_{\ell} \setminus \{i_r\} \label{eq:radial_flow_limit} \\
   &\sum_{(i,j) \in \mathcal{E}_{\ell}} (\lambda_{ij} + \lambda_{ji}) = |\mathcal{N}_{\ell}| - 1 \label{eq:radial_spanning_tree} \\
   &\lambda_{ij} + \lambda_{ji} = \beta_{ij}, \quad \forall (i, j) \in \mathcal{E}_{\ell} \label{eq:radial_directional_edge}
\end{align}

\subsection{Resource Allocation Constraints}
To guarantee operational feasibility, the network must balance resource supply and demand across all active blocks while respecting generation and flow capacities.

\subsubsection{Resource Balance}
For each base block $\ell \in \mathcal{B}$, the net resource flow must equal the total generated resources minus the active consumer demand. Let $r_{ij}^{\mathrm{flow}}$ be the resource flow on edge $(i, j)$, $r_g^{\mathrm{gen}}$ the resource generated by provider $g$, and $D_m$ the demand of consumer $m$. The block-level balance equation is:
\begin{align}
\sum_{(i, j) \in \mathcal{E}_{\ell}} r_{ij}^{\mathrm{flow}} = \sum_{g \in \mathcal{P}_{\ell}} r_g^{\mathrm{gen}} - z_{\ell}^{\mathrm{bl}} \sum_{m \in \mathcal{M}_\ell} D_m, ~~ \forall \ell \in \mathcal{B}.
\end{align}
This formulation ensures that consumer demand is only enforced when the block is active ($z_{\ell}^{\mathrm{bl}} = 1$).

\subsubsection{Generation Capacity}
Providers $g \in \mathcal{P}_{\ell}$ generate within their physical limits ($C_g^{\min}$, $C_g^{\max}$), and only when their block $\ell$ is active:
\begin{align}
z_{\ell}^{\mathrm{bl}} C_g^{\min} \leqslant r_g^{\mathrm{gen}} \leqslant z_{\ell}^{\mathrm{bl}} C_g^{\max}, \quad \forall g \in \mathcal{P}_{\ell}, \forall \ell \in \mathcal{B}.
\end{align}

\subsubsection{Flow Limits}
Flows must respect line capacities $R_{ij}^{\max}$. Fixed links ($\mathcal{E} \setminus \mathcal{E}^{\mathrm{sw}}$) are always bounded by $R_{ij}^{\max}$, whereas flow on a controllable edge is prohibited unless the switch is closed ($z_{ij}^{\mathrm{sw}} = 1$):
\begin{align}
   |r_{ij}^{\mathrm{flow}}| &\leqslant R_{ij}^{\max}, \quad &&\forall (i, j) \in \mathcal{E} \setminus \mathcal{E}^{\mathrm{sw}}, \label{eq:flow_fixed}\\
   |r_{ij}^{\mathrm{flow}}| &\leqslant z_{ij}^{\mathrm{sw}} R_{ij}^{\max}, \quad &&\forall (i, j) \in \mathcal{E}^{\mathrm{sw}}. \label{eq:flow_switched}
\end{align}

\subsection{Block Status Constraints}
To guarantee that base blocks connected by a closed controllable edge share the same activation state, we impose:
\begin{align}
   |z^{\mathrm{bl}}_{i} - z^{\mathrm{bl}}_{j}| &\leqslant 1 - z^{\mathrm{sw}}_{ij} , \quad \forall (i, j) \in \mathcal{E}^{\mathrm{sw}}.
   \label{eq:cc_consistency}
\end{align}

\subsection{Leader Constraints in Subnetworks}
In many applications, regulating specific node types within each subnetwork ensures system feasibility, stability, and operational efficiency. This arises in power systems (grid-forming inverters), communication networks (data aggregation points), and multi-agent coordination (leader-follower architectures). In this work, we impose such constraints on leader nodes, requiring that every active CC, defined as a maximal set of active base blocks mutually linked by closed switches, contains at least one and at most $\kappa \geqslant 1$ leaders. 

For a given configuration $\boldsymbol{z}^{\mathrm{sw}}$, let
$\mathcal{L}_{\mathcal{C}}$ denote the candidate
leader nodes in each CC $c\in\mathcal{C}(\boldsymbol{z}^{\mathrm{sw}})$. Then activation state $z^{\mathrm{CC}}_c$ must satisfy:
\begin{align}
z^{\mathrm{CC}}_c \leqslant \sum_{g \in \mathcal{L}_{\mathcal{C}}} z^{\mathrm{ldr}}_g \leqslant \kappa \cdot  z^{\mathrm{CC}}_c, \quad \forall c \in \mathcal{C}(\boldsymbol{z}^{\mathrm{sw}}).
\end{align}

Enforcing this constraint directly is challenging because the CC topology $\mathcal{C}(\boldsymbol{z}^{\mathrm{sw}})$ is endogenous to the optimization. To decouple the topology, we adapt a coloring scheme from \cite{moring2024robust} and combine it with flow-based connectivity verification.

\subsubsection{Coloring-Based Leader Detection} We introduce a binary variable $y^{\ell}_{ij} \in \{0, 1\}$ indicating whether switch $(i,j) \in \mathcal{E}^{\mathrm{sw}}$ is ``colored" by base block $\ell \in \mathcal{B}$. 

Coloring is restricted to closed switches, and a colored switch implies the presence of an appropriate number of leaders in its associated block:
\begin{align} 
\sum_{\ell \in \mathcal{B}} y^{\ell}_{ij} &\leqslant z^{\mathrm{sw}}_{ij}, \quad &&\forall (i,j) \in \mathcal{E}^{\mathrm{sw}}, \label{eq:coloring_assignment}\\ 
y^{\ell}_{ij} &\leqslant \sum_{g \in \mathcal{P}_{\ell}} z^{\mathrm{ldr}}_g, \quad &&\forall \ell \in \mathcal{B}, \forall (i,j) \in \mathcal{E}^{\mathrm{sw}}. \label{eq:coloring_consistency} 
\end{align}
Furthermore, if a closed switch is colored by block $\ell$, the block's leader count is strictly bounded:
\begin{align}
   y^{\ell}_{ij} - \big(1 - z^{\mathrm{sw}}_{ij}\big)
   \leqslant
   \sum_{g \in \mathcal{P}_{\ell}} z^{\mathrm{ldr}}_g
   &\leqslant
   \kappa
   \big(
   y^{\ell}_{ij} + \big(1 - z^{\mathrm{sw}}_{ij}\big)
   \big),\nonumber\\
   &\forall \ell \in \mathcal{B},~
   \forall (i,j) \in \mathcal{E}^{\mathrm{sw}}.
   \label{eq:leader_participation}
\end{align}
To maintain CC identification, all closed switches within the same CC must share consistent coloring for any block $\ell'$. For any two distinct switches $(a,b)$ and $(d,c)$ in the same CC, we have $\forall \ell' \in \mathcal{B},~
   \forall (a,b)\neq (d,c)\in \mathcal{E}^{\mathrm{sw}}:$
\begin{align}
   \left\{
   \begin{array}{l}
       y^{\ell'}_{dc} - \big[(1 - z^{\mathrm{sw}}_{dc}) + (1 - z^{\mathrm{sw}}_{ab})\big]
       \;\leqslant\; y^{\ell'}_{ab}, \\
       y^{\ell'}_{ab} \;\leqslant\;
       y^{\ell'}_{dc} + \big[(1 - z^{\mathrm{sw}}_{dc}) + (1 - z^{\mathrm{sw}}_{ab})\big]
   \end{array}
   \right. .
   \label{eq:coloring_sync}
\end{align}
Finally, block activation is dependent on either containing a local leader or connecting to an active CC:
\begin{align}
   z^{\mathrm{bl}}_{\ell} \leqslant \sum_{g \in \mathcal{P}_{\ell}} z^{\mathrm{ldr}}_g + \sum_{(i,j) \in \mathcal{E}^{\mathrm{sw}}_{\ell}} \sum_{\ell' \in \mathcal{B}} y^{\ell'}_{ij}, \quad \forall \ell \in \mathcal{B}.
   \label{eq:block_activation_leader}
\end{align}

\subsubsection{Flow-Based Verification of Leader Connectivity}
Let $\mathcal{E}^{v}_{\ell}$ be the set of virtual edges directed from block $\ell$ to all other blocks $\ell'$. We define $\xi^{\ell}_{ij} \in [0,1]$ as the flow across virtual edge $(i,j)$, and $\eta^{\ell}_{ij}$ as the flow magnitude across a closed switch $(i,j) \in \mathcal{E}^{\mathrm{sw}}_{\ell}$. The flows are constrained by switch capacity and network balance:
\begin{align}
&- z^{\sw}_{ij} |\mc{E}^{\sw}| \leqslant \eta^{\ell}_{ij} \leqslant z^{\sw}_{ij} |\mc{E}^{\sw}|, ~ \forall (i,j) \in \mc{E}^{\sw}, \forall \ell \in \mc{B}, \! \label{eq:gfder_flow_limitation_switches}\\
&\!\sum_{\substack{(i,j) \in \mc{E}^{\sw}_{\ell} \\ i = \ell}} \!\eta^{\ell}_{ij} - \!\!\! \sum_{\substack{(i,j) \in \mc{E}^{\sw}_{\ell} \\ j = \ell}} \!\eta^{\ell}_{ij} +\!\!\! \sum_{(i,j) \in \mc{E}^{v}_{\ell}} \!\xi^{\ell}_{ij} = |\mc{B}| - \!1, \forall \ell \in \mc{B}, \! \label{eq:leader_flow_balance_within} \\
&\!\sum_{\substack{(i,j) \in \mc{E}^{\sw}_{\ell} \\ i = \ell'}} \!\eta^{\ell}_{ij} - \!\!\sum_{\substack{(i,j) \in \mc{E}^{\sw}_{\ell} \\ j = \ell'}} \!\eta^{\ell}_{ij} - \xi^{\ell}_{\ell\ell'} = -1,  \forall \ell' \neq \ell, \forall \ell \in \mc{B}. \!\! \label{eq:leader_flow_balance_distinct}
\end{align}

Virtual flow restricts the coloring assignment, ensuring that if flow exists on the virtual line between $\ell$ and $\ell'$, the corresponding switches cannot be colored by $\ell$:
\begin{align} y^{\ell}_{ij} \leqslant 1 - \xi^{\ell}_{\ell\ell'}, \quad \forall \ell' \neq \ell, \quad \forall (i,j) \in \mc{E}^{\sw}_{\ell'}, \quad \forall \ell \in \mc{B} \label{eq:leader_flow_restriction_virtual}
\end{align}
\subsubsection{Leader Constraints within Individual Blocks}
Isolated active blocks must independently satisfy the leader bounds:
\begin{align} z^{\mathrm{bl}}_{\ell} - \!\sum_{(i,j) \in \mathcal{E}^{\mathrm{sw}}_{\ell}} \!\!\! z^{\mathrm{sw}}_{ij} \leqslant \sum_{g \in \mathcal{P}_{\ell}} z^{\mathrm{ldr}}_g \leqslant \kappa \cdot z^{\mathrm{bl}}_{\ell}, \quad \forall \ell \in \mathcal{B}. \label{eq:block_leader_limit} 
\end{align}

\section{Iterative Cutting-Plane Framework} \label{sec:algo}
This section introduces a scalable framework for optimal network partitioning. Since directly embedding combinatorial constraints in a monolithic MIP is computationally intractable (sec. \ref{sec:constraints}), we propose a dynamic constraint enforcement strategy that reduces the computational burden while preserving optimality.

\subsection{Algorithmic Framework}
The cutting-plane method (in Algorithm~\ref{alg:cutting_plane})  is an iterative refinement procedure that enforces complex constraints dynamically. Let $\Omega$ denote the original problem over the full constraint set $\Xi$. We first construct a relaxation $\Omega^{(0)}$ with constraint set $\Xi^{(0)}_\text{relax}$, which retains the computationally tractable operational requirements (resource balance, flow limits, generation capacities, and block status) while deferring the combinatorial radiality and leader-allocation constraints to the iterative phase. If the resulting candidate $\theta^{(0)}$ satisfies all constraints in $\Xi$, it is optimal; otherwise, the violated constraints $\Xi_\text{viol}^{(0)} \subseteq \Xi$ are identified and cutting planes $\Lambda^{(0)}$ are generated to exclude the infeasible candidate while preserving every point of $\Xi$. The problem is updated as $\Omega^{(1)} = \Omega^{(0)} \cup \Lambda^{(0)}$, and the process repeats until all constraints are satisfied.

\begin{algorithm}[t!]
\caption{Iterative Cutting--Plane Method}
\label{alg:cutting_plane}
\SetAlgoLined
\hspace{-.5em}\KwIn{Optimization problem $\Omega$ with constraint set $\Xi$}
\hspace{-.5em}\KwOut{Optimal solution $\theta^{*}$}
\hspace{-.5em}Construct relaxed problem $\Omega^{(0)}$ with constraints $\Xi^{(0)}_\text{relax}$\;
\hspace{-.5em}Initialize iteration index $\tau \gets 0$\;
\Repeat{a solution satisfying all constraints is obtained}{
   Solve $\Omega^{(\tau)}$ to obtain candidate solution $\theta^{(\tau)}$\;
   \eIf{$\theta^{(\tau)}$ satisfies all constraints in $\Xi$}{
       \Return $\theta^{(\tau)}$ \footnotesize{\tcp{Optimal solution found}}
   }{
       Identify violated constraints $\Xi_\text{viol}^{(\tau)} \subseteq \Xi$\;
       Generate cutting planes $\Lambda^{(\tau)}$ to eliminate $\theta^{(\tau)}$ \!\!\;\!\!
       Update problem: $\Omega^{(\tau+1)} \gets \Omega^{(\tau)} \cup \Lambda^{(\tau)}$\;
   }
   $\tau \gets \tau + 1$\;
}
\end{algorithm}

\subsection{Cycle Elimination via Cutting-Planes}
\label{subsec:cycle_cuts}

Directly embedding \eqref{eq:radial_flow_conservation}--\eqref{eq:radial_directional_edge} is computationally expensive, so we initially relax these constraints and enforce radiality dynamically.

\subsubsection{Cycle Detection}
In each iteration, the candidate topology defined by the binary switch variables $z^{\sw}_{ij}$ is inspected using standard graph-theoretic algorithms (e.g., depth-first search or union-find). If any cycles are identified, we systematically generate valid inequalities to eliminate them. 

\subsubsection{Valid Cuts for Cycle Elimination}
Let $C_{\text{cycle},c} \subseteq \mathcal{E}^{\sw}$ denote the set of controllable switches forming a detected cycle $c \in \{1, \dots, C\}$. To break the cycle without removing valid radial topologies, we impose the following constraint:
\begin{align}
   \sum_{(i, j) \in C_{\text{cycle}, c}} z_{ij}^{\sw}
   \;\leqslant\; |C_{\text{cycle}, c}| - 1,
   \quad \forall c = 1,\dots,C.
   \label{eq:radiality_cut_multiple}
\end{align}
Enumerating all graph cycles is NP-hard, so we use a fundamental cycle basis. Since each base block is internally radial with fixed edges, radiality is enforced on the block graph  $\mathbb{G}_{\mathcal{B}}=(\mathcal{B},\mathcal{E}_{\mathrm{cl}}^{\mathrm{sw}})$ induced by the candidate closed switches $\boldsymbol{z}^{\mathrm{sw}}$. A depth-first search constructs a basis of size $|\mathcal{E}_{\mathrm{cl}}^{\mathrm{sw}}|-|\mathcal{B}|+|\mathcal{C}(\boldsymbol{z}^{\mathrm{sw}})|$ in $\mathcal{O}(|\mathcal{B}|+|\mathcal{E}_{\mathrm{cl}}^{\mathrm{sw}}|)$ time. The candidate is radial if and only if this basis is empty, providing an exact radiality certificate, while accumulated cuts \eqref{eq:radiality_cut_multiple} ensure finite progress.

\subsection{Leader Allocation via Cutting-Planes}
\label{subsec:gf_cuts}
Directly enforcing the full leader constraints \eqref{eq:coloring_assignment}--\eqref{eq:leader_flow_restriction_virtual} is computationally demanding, so we initially relax these constraints and enforce the leader bounds iteratively.

\subsubsection{Leader Violation Detection}
Given a candidate solution $(z_{\ell}^{\mathrm{bl},*}, z^{\mathrm{sw},*}_{ij}, z_{g}^{\mathrm{ldr},*})$, the components $\mathcal{C}(\boldsymbol{z}^{\mathrm{sw}})$ are enumerated by connected-component traversal on the block graph with vertex set $\mathcal{B}$ and edge set given by the incumbent-closed switches, at cost $\mathcal{O}(|\mathcal{B}|+|\mathcal{E}^{\mathrm{sw}}|)$. By construction, each component is internally connected and bounded entirely by open switches, and \eqref{eq:cc_consistency} forces its blocks to share a common activation state; the only remaining condition is the leader bound $1 \leqslant \sum_{g\in \mathcal{L}_{\mathcal{C}}} z^{\mathrm{ldr}}_{g} \leqslant \kappa$, where $\mathcal{L}_{\mathcal{C}}$ is the set of providers in $\mathcal{C}$. Any active component violating this bound triggers the cuts that exclude the infeasible candidate solution.

\subsubsection{Valid Cuts for Leader Requirements}
Let $S_{\mathcal{C}}^{\sw,\ex}$ denote the set of external switches around $\mathcal{C}$, $S_{\mathcal{C}}^{\sw,\text{in}}$ denote the closed internal switches within $\mathcal{C}$, and $S_{\mathcal{C}}^{\bl,\text{in}}$ denote the set of base blocks in $\mathcal{C}$. For every CC $\mathcal{C}$, the following constraints establish the bounds for the number of leaders in CC $\mathcal{C}$:
\begin{subequations}
\begin{align}
\Phi(\mathcal{C}) \;:=\;& 
    \prod_{(i,j)\,\in\, S_{\mathcal{C}}^{\mathrm{sw,ex}}} \!\!(1 - z_{ij}^{\mathrm{sw}}) 
    \cdot \prod_{(i,j)\,\in\, S_{\mathcal{C}}^{\mathrm{sw,in}}} \!\! z_{ij}^{\mathrm{sw}} 
    \cdot \prod_{\ell\,\in\, S_{\mathcal{C}}^{\mathrm{bl,in}}} z_{\ell}^{\mathrm{bl}}, 
    \nonumber \\
\sum_{g \in \mathcal{L}_\mathcal{C}} z_g^{\mathrm{ldr}} 
    &\;\geqslant \; \Phi(\mathcal{C}), 
    \label{eq:lb_leader_req} \\
\sum_{g \in \mathcal{L}_\mathcal{C}} z_g^{\mathrm{ldr}} 
    &\;\leqslant \; |\mathcal{L}_\mathcal{C}| \;+\; 
    \big(\kappa - |\mathcal{L}_\mathcal{C}|\big)\cdot \,\Phi(\mathcal{C}).
    \label{eq:ub_leader_req}
\end{align}
\label{eq:leader_req_nonlin}
\end{subequations}

\subsubsection{Linearization of \eqref{eq:leader_req_nonlin}}
Constraints \eqref{eq:leader_req_nonlin} contain products of binary variables, which can be linearized exactly using standard auxiliary-variable formulations. Instead, we exploit their special structure and use the following linear bounds, which are exact without additional variables:
\begin{subequations}
\begin{align}
    \sum_{g \in \mc{L}_{\mc{C}}} z^{\mathrm{ldr}}_g \geqslant &\sum_{(i, j) \in S_{\mc{C}}^{\sw,\ex}} (1-z^{\sw}_{ij}) + \sum_{(i,j) \in S_\mathcal{C}^{\mathrm{sw,in}}} z_{ij}^{\mathrm{sw}} +\nonumber \\
    &\sum_{\ell \in S_{\mc{C}}^{\bl,\text{in}}} z^{\bl}_{\ell}   - |S_{\mc{C}}^{\sw,\ex}| -|S_{\mc{C}}^{\sw,\text{in}}| -  |S_{\mc{C}}^{\bl,\text{in}}| + 1, \label{eq:inv_cut_lower_bound}\\
    \sum_{g \in \mc{L}_{\mc{C}}} z^{\mathrm{ldr}}_g \leqslant \  & \kappa + \left( |\mc{L}_{\mc{C}}| - \kappa \right) \Bigg( \sum_{(i, j) \in S_{\mc{C}}^{\sw,\ex}} z_{ij}^{\sw} \nonumber\\
    +& \sum_{(i, j) \in S_{\mc{C}}^{\sw,\text{in}}} (1-z_{ij}^{\sw}) + \sum_{\ell \in S_{\mc{C}}^{\bl,\text{in}}} (1-z^{\bl}_{\ell}) \Bigg). \label{eq:inv_cut_upper_bound} 
\end{align}
\label{eq:valid_cuts_inv_linear}
\end{subequations}

\subsection{Theoretical Guarantees}
This subsection formally proves that the dynamically generated cutting planes preserve the feasible region (from Sec.~\ref{sec:constraints}) and that the overall algorithm guarantees finite convergence to a global optimum.

\begin{theorem}[Validity of Cycle-Elimination Cuts] \label{thm:cycle_validity}
Let $C_{\text{cycle},c} \subseteq \mathcal{E}^{\mathrm{sw}}$ be a set of controllable edges that form a detected cycle $c$ in the candidate solution. The cutting-plane constraint 
in \eqref{eq:radiality_cut_multiple}
eliminates the infeasible cyclic configuration without excluding any valid radial topology.
\end{theorem}

\begin{proof}
Consider a cycle $C_{\text{cycle},c}$ in which every switch $(i,j)\in C_{\text{cycle},c}$ is closed, implying
\[\sum_{(i,j)\in C_{\text{cycle},c}} z_{ij}^{\sw} \;=\; |C_{\text{cycle},c}|.\]
Such a configuration violates the radiality requirement. To ensure at least one switch in this cycle is opened, we enforce
\[
\sum_{(i,j)\in C_{\text{cycle},c}} z_{ij}^{\sw} \;\leqslant\; |C_{\text{cycle},c}| - 1.
\]
This inequality disallows any solution in which \textit{all} switches in $C_{\text{cycle},c}$ are set to $1$. Consequently, the cycle is broken, restoring radiality. Furthermore, this cut excludes only those solutions that contain a fully closed cycle, thereby leaving all already-feasible (radial) solutions unaffected. Hence, \eqref{eq:radiality_cut_multiple} is a valid cycle-elimination constraint.
\end{proof}

\begin{theorem}[Validity of Leader Allocation Cuts] \label{lem:leader_validity}
The generated cutting-plane constraints for leader allocation, as formulated in \eqref{eq:leader_req_nonlin}, correctly enforce the hierarchical requirements for active, internally connected components without eliminating any valid network configurations.
\end{theorem}

\begin{proof}
Suppose that the CC $\mathcal{C}$ is active and fully isolated, meaning that every block in $S_{\mathcal{C}}^{\bl,\text{in}}$ is active ($z_{\ell}^{\bl}=1$) and every external switch in $S_{\mathcal{C}}^{\sw,\ex}$ is open ($z_{ij}^{\sw}=0$), and that $\mathcal{C}$ is internally connected, so every internal switch in $S_{\mathcal{C}}^{\sw,\text{in}}$ is closed ($z_{ij}^{\sw}=1$). Under these conditions, each factor in $\Phi(\mathcal{C})$ equals $1$, so $\Phi(\mathcal{C})=1$. Consequently, the constraints in \eqref{eq:leader_req_nonlin} simplify to
$
1 \leqslant \sum_{g \in \mathcal{L}_{\mathcal{C}}} z_{g}^{\mathrm{ldr}} \leqslant \kappa.
$
Thus, for an active, fully isolated, and internally connected CC, these inequalities enforce that the number of leaders is between $1$ and $\kappa$.

Otherwise, at least one factor of $\Phi(\mathcal{C})$ vanishes, so $\Phi(\mathcal{C})=0$ and \eqref{eq:leader_req_nonlin} reduces to
$
0 \leqslant \sum_{g \in \mathcal{L}_{\mathcal{C}}} z_{g}^{\mathrm{ldr}} \leqslant |\mathcal{L}_{\mathcal{C}}|,
$
which is trivially satisfied since each $z_{g}^{\mathrm{ldr}}$ is binary. Thus, these constraints are binding only on assignments that reproduce $\mathcal{C}$ exactly, and exclude only solutions that violate the corresponding leader requirements, completing the proof.
\end{proof}

\begin{lemma}\label{lem:exact_linearization}
\textit{(Exactness of Linearization)} For any component $\mathcal{C}$ with $|\mathcal{L}_{\mathcal{C}}| \geqslant \kappa$, the linear cuts in \eqref{eq:valid_cuts_inv_linear} are exactly equivalent to the nonlinear constraints in \eqref{eq:leader_req_nonlin}.
\end{lemma}

\begin{proof}
For the lower bound \eqref{eq:inv_cut_lower_bound}, let $N = \sum_{(i, j) \in S_{\mc{C}}^{\sw,\ex}} (1-z^{\sw}_{ij}) + \sum_{(i,j) \in S_{\mc{C}}^{\sw,\text{in}}} z^{\sw}_{ij} + \sum_{\ell \in S_{\mc{C}}^{\bl,\text{in}}} z^{\bl}_{\ell}$ and $n_{\mc C} = |S_{\mc{C}}^{\sw,\ex}| + |S_{\mc{C}}^{\sw,\text{in}}| + |S_{\mc{C}}^{\bl,\text{in}}|$. Since all variables are binary, $N = n_{\mc C}$ if and only if every factor in $\Phi(\mc C)$ equals $1$; otherwise $N \leqslant n_{\mc C}-1$. Hence the right-hand side of \eqref{eq:inv_cut_lower_bound} equals $1$ when $\Phi(\mc C)=1$, and is at most $0$ when $\Phi(\mc C)=0$. Therefore, \eqref{eq:inv_cut_lower_bound} becomes $\sum_{g \in \mc L_{\mc C}} z_g^{\mathrm{ldr}} \geqslant 1$ when $\Phi(\mc C)=1$, and is non-binding when $\Phi(\mc C)=0$. Thus, \eqref{eq:inv_cut_lower_bound} is exactly equivalent to \eqref{eq:lb_leader_req} for binary assignments.

For the upper bound \eqref{eq:inv_cut_upper_bound}, let $V = \sum_{(i, j) \in S_{\mc{C}}^{\sw,\ex}} z_{ij}^{\sw} + \sum_{(i, j) \in S_{\mc{C}}^{\sw,\text{in}}} (1-z_{ij}^{\sw}) + \sum_{\ell \in S_{\mc{C}}^{\bl,\text{in}}} (1-z^{\bl}_{\ell})$. Since all variables are binary, $V = 0$ if and only if $\Phi(\mathcal{C}) = 1$; otherwise $V \geqslant 1$. When $\Phi(\mathcal{C}) = 1$, the RHS reduces to $\kappa$, enforcing $\sum_{g \in \mathcal{L}_\mathcal{C}} z_g^{\mathrm{ldr}} \leqslant \kappa$ as required. When $\Phi(\mathcal{C}) = 0$, we have $V \geqslant 1$, and since $|\mathcal{L}_\mathcal{C}| \geqslant \kappa$ the coefficient $(|\mathcal{L}_\mathcal{C}| - \kappa)$ is nonnegative, so $\kappa + (|\mathcal{L}_\mathcal{C}| - \kappa)V \geqslant |\mathcal{L}_\mathcal{C}|$ and the constraint is non-binding since $\sum_{g \in \mathcal{L}_\mathcal{C}} z_g^{\mathrm{ldr}} \leqslant |\mathcal{L}_\mathcal{C}|$ always holds for binary variables. Thus \eqref{eq:inv_cut_upper_bound} is exactly equivalent to \eqref{eq:ub_leader_req}. The hypothesis always holds in practice: \eqref{eq:inv_cut_upper_bound} is separated only when an integer-feasible candidate violates the upper leader bound, which requires $|\mathcal{L}_\mathcal{C}| > \kappa$. 
\end{proof}

\begin{theorem}[Finite Convergence and Global Optimality] \label{thm:convergence}
Algorithm~\ref{alg:cutting_plane} converges to a globally optimal solution of the monolithic network partitioning problem \eqref{eq-OptimalNetworkPartitioningProblem} in a finite number of iterations, provided that a feasible solution exists.
\end{theorem}

\begin{proof}
Let $\mathcal{Z}$ denote the set of integer assignments admitted by the initial relaxed master problem, comprising $(\boldsymbol{z}^{\mathrm{sw}},\boldsymbol{z}^{\mathrm{bl}},\boldsymbol{z}^{\mathrm{ldr}})$ and any other binary variables retained in the relaxation. Since the number of such binary variables is finite, so is $\mathcal{Z}$. The continuous variables do not affect this count: for any fixed integer assignment, the remaining problem is a linear program with an attained optimum if feasible and bounded.

At any iteration $\tau$, the relaxed problem $\Omega^{(\tau)}$ is subject to a constraint set $\Xi^{(\tau)} \subseteq \Xi$. Because $\Omega^{(\tau)}$ is a relaxation, its optimal objective value provides a valid lower bound for the monolithic problem $\Omega$. Let $\theta^{(\tau)}$ be the \textit{integer-feasible} candidate solution obtained by solving $\Omega^{(\tau)}$ to optimality. If $\theta^{(\tau)}$ satisfies all constraints in the full set $\Xi$, it is feasible for $\Omega$. Given that $\Omega^{(\tau)}$ is a relaxation, $\theta^{(\tau)}$ must therefore be a global optimum of $\Omega$.

If $\theta^{(\tau)}\notin\Xi$, cuts are separated at the current integer-feasible candidate of the relaxation. Each generated cut is violated by that candidate's integer assignment and is satisfied by every point of $\Xi$ by Theorems~\ref{thm:cycle_validity}--\ref{lem:leader_validity}. Thus, that assignment is permanently excluded from subsequent relaxations. A cut may also exclude other infeasible assignments, but only exclusion of the current candidate is required for finiteness. Thus, each cut-generating iteration removes at least one assignment from the finite set $\mathcal{Z}$, so only finitely many such iterations are possible. Since the separation checks are exact and no cut removes a point of $\Xi$, the algorithm terminates at a feasible candidate whenever one exists; optimality of the relaxation then implies global optimality.
\end{proof}

\section{Numerical Results} \label{sec:results}
This section presents simulation studies in power distribution systems to evaluate the proposed cutting-plane framework. We consider a contingency scenario where the main grid connection is lost, requiring the network to transition into one or more self-sustained microgrids (see \cite{moring2024robust}).

\subsection{Case study}
\label{subsec:case_study}
The Iowa 240-bus system \cite{bu2019time} used in this study is an unbalanced, three-phase distribution feeder with strategically placed switchable lines. The original feeder has 240 buses---130 single-phase, 1 two-phase, and 109 three-phase---and $9$ controllable switches. We disable battery storage and add $37$ additional dispatchable switches (with no other network modifications), yielding up to $46$ controllable switches; the reported studies use the $26$, $36$, and $46$ switch cases. The network hosts $5$ DERs with a total generation capacity of $2300\,\text{kW}$, unevenly distributed to reflect practical resource allocation. Additionally, $192$ loads are distributed across the network's nodes and phases, representing phase imbalance typical of real-world feeders, with an aggregate nominal load of $1167.3\,\text{kW}$. A disruptive event at the upstream substation isolates the feeder from the main grid, forcing the feeder to operate in islanded mode and potentially partition into multiple microgrids controlled by the proposed reconfiguration strategy. Demand variability is captured by independently sampling each load multiplier from the uniform distribution on $[0.8,1.0]$. We evaluate $250$ seeds, each as an independent deterministic instance; consequently, these experiments do not provide a robust or probabilistic guarantee.

The objective balances load shedding and generation costs for serving energized demand. Let $\nu$ weight load shedding, with generation weighted by $1-\nu$. To uphold standard protection practices, the feeder is constrained to be radial, and each energized subnetwork must include exactly one grid-forming DER ($\kappa=1$), with others in grid-following mode. Block priorities are weighted as $\alpha_{\ell}=\gamma|\mathcal{D}_{\ell}|$, assigning a higher cost to shedding larger load blocks.

\subsection{Simulation Setup}
As discussed earlier, explicitly enforcing radiality and grid-forming constraints introduces many auxiliary variables and flow constraints, causing the complexity of finding a feasible topology to scale exponentially with the number of controllable edges. For each sampled demand realization, the optimization also enforces linearized three-phase \textit{LinDist3Flow} power-flow constraints \cite{fobes2022optimal}, along with radiality and grid-forming constraints. To demonstrate the effectiveness of the proposed approach, we solve the network reconfiguration problem under the following scenarios:
\begin{list}{$\bullet$}{\leftmargin=0.5em \itemindent=0pt \itemsep=2pt}
    \item \textbf{Full-MIP}: MIP formulation with full radiality and grid-forming 
    constraints (Sec.~\ref{sec:constraints}).
    \item \textbf{CP-Radial}: Cutting-plane method that enforces radiality 
    (Sec.~\ref{subsec:cycle_cuts}) with full grid-forming MIP constraints 
    (Sec.~\ref{sec:constraints}).
    \item \textbf{CP-GF}: Cutting-plane method that enforces grid-forming 
    (Sec.~\ref{subsec:gf_cuts}) with full radiality MIP constraints 
    (Sec.~\ref{sec:constraints}).
    \item \textbf{CP-Radial+GF}: Cutting-plane method that enforces both radiality 
    and grid-forming constraints (Sec.~\ref{subsec:cycle_cuts}, 
    \ref{subsec:gf_cuts}).
\end{list}
\noindent
\textbf{Computational setup:} All formulations, including Full-MIP and the proposed cutting-plane variants (CP-Radial, CP-GF, and CP-Radial+GF), are implemented within PowerModelsONM.jl~\cite{fobes2022optimal}, available at \url{https://github.com/lanl-ansi/PowerModelsONM.jl}. All experiments are conducted on a MacBook with an M4 Max chip (14 cores, 36\,GB memory). The MIP formulations are solved using Gurobi v13.0.3 with default solver settings, interfaced via JuMP v1.31.2 in Julia v1.12.7. The cutting-plane variants use Gurobi’s lazy-constraint callback feature.

\subsection{Performance of Proposed Approach}
\subsubsection{Cut Statistics Across Increasing Switch Configurations}
\begin{table}[t]
\vspace{.3cm}
\centering
{\setlength{\tabcolsep}{4pt}
\begin{tabular}{l|l|rrr|rrr}
\hline
\multirow{2}{*}{\shortstack{\#\\switches}}
& \multirow{2}{*}{Method}
& \multicolumn{3}{c|}{\textbf{Radial Cuts}}
& \multicolumn{3}{c}{\textbf{GF Cuts}} \\
& & Avg. & Min. & Max. & Avg. & Min. & Max. \\
\hline
\multirow{3}{*}{26}
& CP-Radial    & 4.3 & 0 & 6 & 0   & 0 & 0 \\
& CP-GF        & 0   & 0 & 0 & 9.2 & 0 & 53 \\
& CP-Radial+GF & 1.0 & 0 & 6 & 20.3 & 2 & 55 \\
\hline
\multirow{3}{*}{36}
& CP-Radial    & 2.5 & 0 & 6 & 0   & 0 & 0 \\
& CP-GF        & 0   & 0 & 0 & 4.3 & 0 & 31 \\
& CP-Radial+GF & 3.9 & 0 & 6 & 1.7 & 0 & 20 \\
\hline
\multirow{3}{*}{46}
& CP-Radial    & 3.1 & 0 & 6 & 0   & 0 & 0 \\
& CP-GF        & 0   & 0 & 0 & 5.0 & 0 & 39 \\
& CP-Radial+GF & 5.6 & 1 & 6 & 0.1 & 0 & 13 \\
\hline
\end{tabular}
}
\caption{Radial and grid-forming (GF) cut statistics across 250 load parameter
realizations for each cutting-plane method and switch configuration.}
\label{tab:iowa240-cut-stats}
\end{table}

Table~\ref{tab:iowa240-cut-stats} highlights two key observations. First, CP-Radial generates very few radiality cuts (avg. 2.5--4.3), suggesting that the GFM coloring MIP constraints already restrict the feasible region enough that radiality violations are rarely encountered during branch-and-cut. Second, CP-Radial+GF requires more GF cuts than CP-GF at 26 switches (20.3 vs. 9.2 avg.), yet this ordering reverses and widens with increasing switching combinatorics (1.7 vs. 4.3 at 36 switches; 0.1 vs. 5.0 at 46). This indicates that the radiality cuts in CP-Radial+GF actively tighten the search space and steer the solver toward GF-feasible solutions, substantially reducing the burden of adding GF cuts. This structural interaction between the two constraint families underlies the solve-time advantage of CP-Radial+GF discussed next.

\subsubsection{Solve-time Efficiency Under Switching 
Flexibility} 
\begin{figure}[h]
    \centering
    \includegraphics[width=1\linewidth]{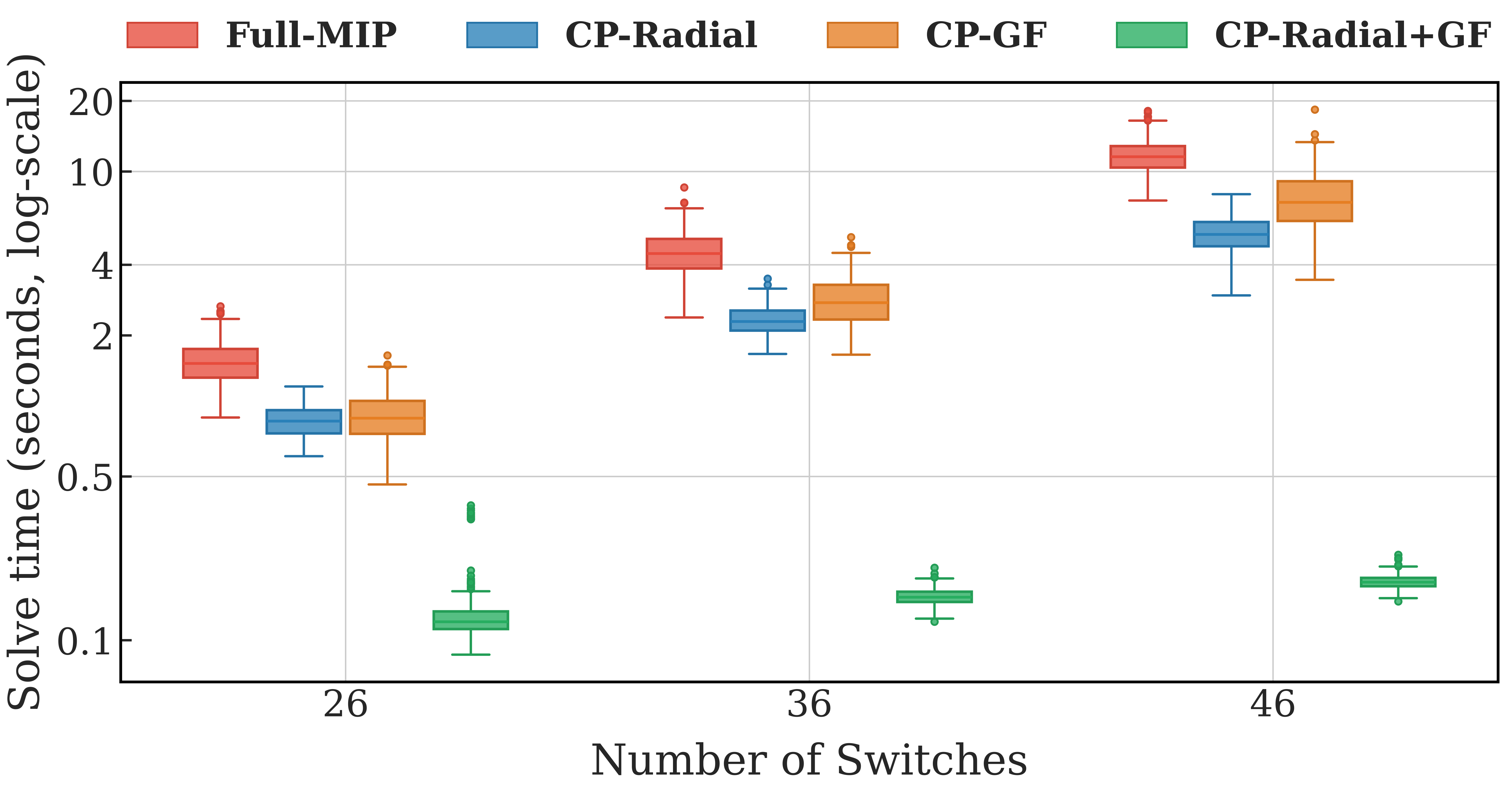}
    \caption{Solve-time distributions for the Iowa-240 network under a substation contingency across 250 random demand realizations and varying dispatchable switch counts. CP-Radial+GF achieves substantially lower solve times than Full-MIP and single-family cut variants, with superior empirical scaling as switching flexibility increases.}
    \label{fig:solve_times_iowa240}
\end{figure}

Figure~\ref{fig:solve_times_iowa240} reports solve-time distributions (in log-scale) across 250 random demand realizations for each formulation under a fixed substation contingency. Median solve times for Full-MIP increase from 1.52\,s at 26 switches to 11.58\,s at 46 switches (7.63$\times$; worst-case: 2.66\,s and 18.13\,s, respectively), consistent with combinatorial growth in the binary feasibility space as switching flexibility increases. CP-GF exhibits nearly identical scaling (8.34$\times$, 0.89\,s to 7.40\,s), indicating that with radiality retained as commodity-flow constraints, radial-feasibility combinatorics remain the dominant bottleneck. CP-Radial partially alleviates this, achieving 6.26$\times$ growth (0.86\,s to 5.40\,s), while requiring only 2.5--4.3 radiality cuts on average across switch configurations. By contrast, CP-Radial+GF shows near-flat empirical scaling of 1.47$\times$ (0.120\,s to 0.177\,s), \textit{with median speedups of 12.6$\times$, 29.2$\times$, and 65.4$\times$ over Full-MIP at 26, 36, and 46 switches, respectively; at the 95th percentile, speedup reaches 78$\times$ at 46 switches}. These results indicate that decomposing only one constraint family is useful but insufficient for scalability, whereas jointly decomposing radiality and GFM coloring constraints removes the combinatorial scaling bottleneck.

\subsubsection{Optimal Post-Contingency Reconfiguration}
As shown in Figure~\ref{fig:IOWA_240_Colored}, under the nominal-load Iowa 240-bus case with 46 controllable switches and a substation contingency, the optimal solution energizes 21 out of 44 load blocks, yielding a radial post-contingency topology. The reconfiguration closes 20 of the 46 controllable switch edges and activates one normally-open tie line, \texttt{cb\_303}, to reroute power around the outaged substation-supplied area. A total of 353.9\,kW of generation is dispatched to serve the energized blocks, satisfying 353.9\,kW out of 1167.3\,kW of total network demand (30.3\%), with the remaining 813.4\,kW (69.7\%) shed across 23 blocks. The substantial load shedding reflects the severity of the substation contingency, which eliminates the dominant generation source and forces the reconfiguration to operate under significant resource constraints.

\begin{figure}
    \vspace{.2cm}
    \centering
    \includegraphics[width=1.0 \linewidth]{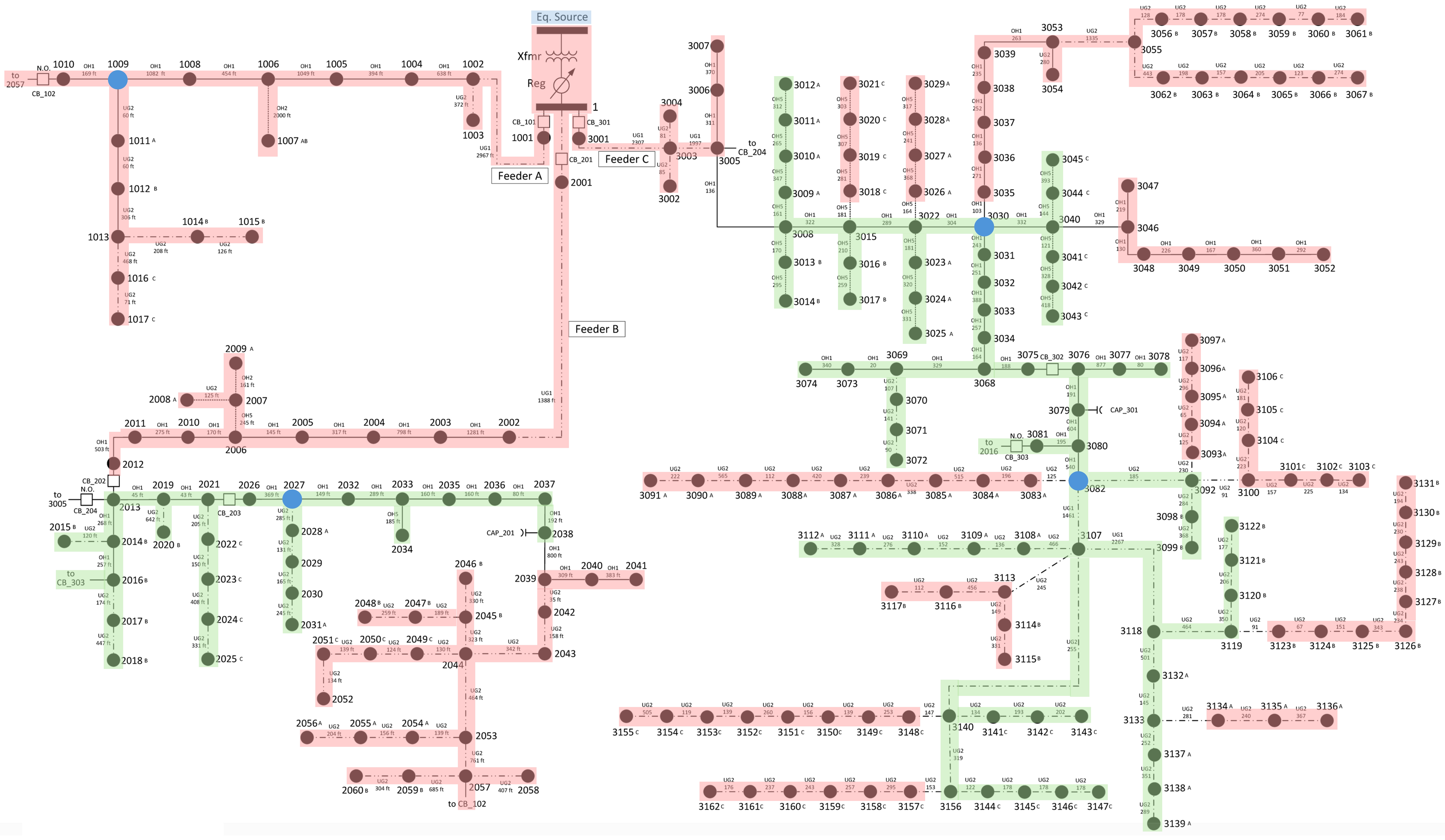}
     \caption{Optimal post-contingency reconfiguration of the Iowa 240-bus system. Energized load blocks are highlighted in green, while de-energized (shed) blocks are shown in red. Blue dots indicate DER locations.}
    \label{fig:IOWA_240_Colored}
\end{figure}

\begin{table}[htbp]
\centering
\begin{tabular}{c|ccc}
\hline
\# Controllable Switches & 26 & 36 & 46 \\
\hline
Load Blocks Energized     & 11/24 & 14/34 & 21/44 \\
Load Served (kW)          & 342.1 & 347.5 & 353.9 \\
\textbf{\% Load Served} & \textbf{29.30\%} & \textbf{29.77\%} & \textbf{30.32\%} \\
\hline
\end{tabular}
\caption{Post-contingency load recovery across switch configurations for Iowa 240 system under nominal demand.}
\label{tab:iowa240-load-served}
\end{table}
Table~\ref{tab:iowa240-load-served} shows that increasing switching flexibility improves post-contingency restoration, with served demand monotonically rising from 29.30\% (26 switches) to 30.32\% (46 switches). While the gains are modest, they reflect limited generation capacity under the substation contingency. However, increasing the switch set significantly raises the combinatorial complexity of the Full MIP, leading to higher solve times (see Fig.~\ref{fig:solve_times_iowa240}). The proposed cutting-plane framework directly addresses this bottleneck, enabling the resilience gains without prohibitive runtime growth.

In summary, the results show that the proposed iterative cutting-plane approach, particularly \textbf{CP-Radial+GF}, significantly accelerates resilient distribution grid reconfiguration under contingencies compared to the Full-MIP, improving the practicality of large-scale network resilience operations. A natural next step is extending the framework to integrated transmission–distribution power networks, where distribution reconfiguration is coupled with transmission constraints, further increasing complexity and emphasizing the need for the proposed scalable algorithms.

\section{Conclusions}
\label{sec:conc}
This paper introduced a cutting-plane framework for efficient graph partitioning under resource constraints, with application to optimal distribution grid reconfiguration under contingency. We investigated key constraints in network partitioning, including radiality, grid-forming (GFM) coloring, and resource balance constraints, and formulated them as a comprehensive mixed-integer programming (MIP) model. Given the computational intractability of solving the full MIP for large-scale networks, we developed novel valid cuts to iteratively enforce these constraints via a lazy-callback cutting-plane approach, \textit{significantly reducing computational complexity while preserving solution optimality}. Numerical experiments on the Iowa 240-bus system under a substation contingency validate the accuracy and efficiency of the proposed framework across varying numbers of dispatchable switches. An ablation study reveals that jointly decomposing both radiality and GFM coloring constraints into cutting planes is critical for scalable graph partitioning---single-family cut variants yield only partial scalability gains, whereas joint GFM+radiality cuts achieve near-flat empirical scaling with median speedups of up to 65.4$\times$ and best speedups exceeding 78$\times$ over full MIP formulations.

\bibliographystyle{IEEEtran}
\bibliography{references}

\end{document}